\documentclass{amsart}

\usepackage{amsmath,amssymb,amsfonts,enumerate,amsthm, amscd}

\newcommand{\pd}{\mbox{pd}\,}
\newcommand{\id}{\mbox{id}\,}
\newcommand{\fd}{\mbox{fd}\,}
\newcommand{\gd}{\mbox{Gdim}\,}

\newtheorem{theorem}{Theorem}[section]
\newtheorem{lemma}[theorem]{Lemma}
\newtheorem{proposition}[theorem]{Proposition}

\theoremstyle{definition}

\newtheorem{definitions}[theorem]{Definitions}
\theoremstyle{remark}
\newtheorem{remark}[theorem]{Remark}

\newtheorem{example}[theorem]{Example}

\theoremstyle{Definition and Notation}

\begin{document}
\bibliographystyle{amsplain}

\title[]{On $n$-strongly Gorenstein rings}

\author{Najib Mahdou}
\address{Najib Mahdou\\Department of Mathematics, Faculty of Science and Technology of Fez, Box 2202, University S.M. Ben Abdellah Fez,
Morocco. \\ mahdou@hotmail.com}

\author{Mohammed Tamekkante}
\address{Mohammed Tamekkante\\Department of Mathematics, Faculty of Science and Technology of Fez, Box 2202, University S.M. Ben Abdellah Fez,
Morocco. \\ tamekkante@yahoo.fr }



\begin{abstract}

This paper introduces and studies a particular subclasses of the
class  of commutative rings with finite Gorenstein global (resp.,
weak) dimensions.
\end{abstract}

\maketitle
\section{Introduction}
Throughout this paper, all rings are commutative with identity,
and
all modules are unitary.\\
Let $R$ be a ring, and let $M$ be an $R$-module. As usual we use
$\pd_R(M)$, $\id_R(M)$ and $\fd_R(M)$ to denote, respectively, the
classical projective dimension, injective dimension and flat
dimension of $M$. We use $gldim(R)$ and $wdim(R)$ to denote,
respectively, the classical global and weak global
 dimension of $R$.\\

For two-sided Noetherian ring $R$, Auslander and Bridger \cite{Aus
bri} introduced the $G$-dimension, $\gd_R (M)$, for every finitely
generated $R$-module $M$. They showed that there is an inequality
$\gd_R (M)\leq \pd_R (M)$ for all finite $R$-modules $M$, and the
equality holds if $\pd_R (M)$ is finite.

Several decades later, Enochs and Jenda \cite{Enochs,Enochs2}
defined the notion of Gorenstein projective dimension
($G$-projective dimension for short), as an extension of
$G$-dimension to modules that are not necessarily finitely
generated, and the Gorenstein injective dimension ($G$-injective
dimension for short) as a dual notion of Gorenstein projective
dimension. Then, to complete the analogy with the classical
homological dimension, Enochs, Jenda and Torrecillas \cite{Eno
Jenda Torrecillas} introduced the Gorenstein flat dimension. Some
references are
 \cite{Bennis and Mahdou2, Christensen, Christensen
and Frankild, Enochs, Enochs2, Eno Jenda Torrecillas, Holm}.\\

Recall that an $R$-module $M$ is called Gorenstein projective if,
there exists an exact sequence of projective $R$-modules:
$$\mathbf{P}:...\rightarrow P_1\stackrel{f_1}\rightarrow P_0\stackrel{f_0}\rightarrow
P^0\stackrel{f^0}\rightarrow P^1\stackrel{f^1}\rightarrow ...$$
such that $M\cong Im(f_0)$ and such that the operator $Hom_R(-,Q)$
leaves $\mathbf{P}$ exact whenever $Q$ is projective. The
resolution
 $\mathbf{P}$ is called a complete projective resolution.
  In particular, if for each $i$ we have $P_i=P^i=P$ and $f^i=f_i=f$, the $R$ module
$M$ is called strongly Gorenstein projective(please see \cite{Bennis and Mahdou1}).\\
The (strongly) Gorenstein injective $R$-modules are defined
dually.\\
 And an $R$-module $M$ is called
Gorenstein flat if, there exists an exact sequence of flat left
$R$-modules:
$$\mathbf{F}:...\rightarrow F_1\stackrel{f_1}\rightarrow F_0\stackrel{f_0}\rightarrow
F^0\stackrel{f^0}\rightarrow F^1\stackrel{f^1}\rightarrow ...$$
such that $M\cong Im(f_0)$ and such that the operator
 $-\otimes_RI$ leaves $F$ exact whenever $I$
is an injective $R$-module. The resolution $\mathbf{F}$ is called
complete flat resolution.  In particular, if for each $i$ we have
$F_i=F^i=F$ and $f^i=f_i=f$, the $R$ module $M$ is called strongly
Gorenstein flat (please see \cite{Bennis and Mahdou1}).

 The  Gorenstein projective, injective and flat
dimensions are defined in term of resolution and denoted by
$Gpd(-)$, $Gid(-)$ and $Gfd(-)$ respectively (see
\cite{Christensen, Enocks and janda, Holm}).

In \cite{Bennis and Mahdou2},  the authors prove the equality:
$$sup\{Gpd_R(M)|M\;is\;an\;R-module\}=sup\{Gid_R(M)|M\;is\;an\;R-module\}$$
They called the common value of the above quantities the
Gorenstein global dimension of $R$ and denoted it by $Ggldim(R)$.
Similarly, they set
$$wGgldim(R)=\{Gfd_R(M)|M\;is\;an\;R-module\}$$
which they called the weak Gorenstein global dimension of $R$.
Note that the results and the notation in \cite{Bennis and
Mahdou2} are in the non-commutative case. In this paper our rings
are commutative and so the left and right (Gorenstein) homological
dimensions can be identified.

Recently, in \cite{MahdouTamekkante}, particular modules of finite
Gorenstein projective, injective and flat dimensions are defined
as follows:
\begin{definitions}
Let $n$ be a positive integer.
\begin{enumerate}
  \item An $R$-module $M$ is said to be strongly $n$-Gorenstein projective,
if there exists a short exact sequence $0\rightarrow M \rightarrow
P \rightarrow M \rightarrow 0$ where $pd_R(P)\leq n$ and
$Ext^{n+1}_R(M,Q)=0$ whenever $Q$ is projective.
  \item An $R$-module $M$ is said to be strongly $n$-Gorenstein injective,
if there exists a short exact sequence $0\rightarrow M \rightarrow
I \rightarrow M \rightarrow 0$ where $id_R(I)\leq n$ and
$Ext^{n+1}_R(E,M)=0$ whenever $E$ is injective.
\item An $R$-module $M$ is said to be
strongly $n$-Gorenstein flat, if there exists a short exact
sequence $0\rightarrow M \rightarrow F \rightarrow M \rightarrow
0$ where $fd_R(F)\leq n$ and $Tor_R^{n+1}(M,I)=0$ whenever $I$ is
injective.
\end{enumerate}
\end{definitions}
Clearly, the strongly $0$-Gorenstein projective, injective and
flat modules are the strongly Gorenstein projective, injective and
flat modules respectively (\cite[Propositions 2.9 and 3.6]{Bennis
and Mahdou2}).

In this paper, we investigate these modules to characterize new
classes of rings with finite (resp., weak) Gorenstein global
dimensions. In \cite{MahdouTamekkante}, the authors prove the
following Proposition:

\begin{proposition}\cite[Proposition 2.16]{MahdouTamekkante}\label{propo defini}
Let $R$ be a ring. The following statements are equivalent:
\begin{enumerate}
  \item Every module is strongly $n$-Gorenstein projective.
  \item Every module is strongly $n$-Gorenstein injective.
\end{enumerate}
\end{proposition}
Hence, we give the following definitions:
\begin{definitions}Let $n$ be a positive integer.
\begin{enumerate}
    \item A ring $R$ is called $n$-strongly Gorenstein ($n$-SG ring for a
short) if R satisfies one of the equivalent conditions of
Proposition \ref{propo defini}.
    \item A ring $R$ is called weakly $n$-strongly Gorenstein
    ($n$-wSG for a short) if every $R$-module is strongly
    $n$-Gorenstein flat.
\end{enumerate}
\end{definitions}
The rings $0$-SG and  $1$-SG are already studied in \cite{ouarghi,
MT1} over which they are called strongly Gorenstein semi-simple
and hereditary rings respectively. Clearly, by definition, every
$n$-SG (resp., $n$-wSG) ring is $m$-SG (resp., $m$-wSG) whenever
$n\leq m$.\\
After given some characterizations of the $n$-SG and $n$-wSG rings
(see Propositions \ref{prop1} and \ref{prop2}), we will see that
for any ring $R$ we have:
$$gldim(R)\leq n\quad \Longrightarrow \;R\; is\; n-SG\Longrightarrow \;Ggldim(R)\leq n,\: and$$
$$wdim(R)\leq n\quad \Longrightarrow \;R\; is\; n-wSG\Longrightarrow \;wGgldim(R)\leq n.$$
We will also give examples which show that the inverse
implications are not true in the general case (see Examples
\ref{1} and \ref{3}). After that, Theorem \ref{thm1} proves that:
$$R\; is\; an\; n-SG\; ring \quad \Longrightarrow \quad R\; is\; an\; n-wSG\; ring$$
with equivalence if $R$ is Noetherian or perfect with finite
Gorenstein global dimension. But in the general case the inverse
implication is not true as shown by Example \ref{3}.
\section{Main results}

Note that the structure of the $n$-SG rings (resp., $n$-wSG rings)
depends of two variants. The first one is the (resp., weak)
Gorenstein global dimension of these rings and the second is the
form of their modules.
\begin{proposition}\label{prop1} For a ring $R$ and a positive integer $n$, the following statements are equivalent:
\begin{enumerate}
    \item $R$ is $n$-SG ring.
    \item $Ggldim(R)\leq n$ and for every $R$-module $M$ there
    exists a short exact sequence $0\rightarrow M \rightarrow
    P\rightarrow M \rightarrow 0$ where $pd_R(P)<\infty$.
    \item $Ggldim(R)<\infty$ and for every $R$-module $M$ there
    exists a short exact sequence $0\rightarrow M \rightarrow
    P\rightarrow M \rightarrow 0$ where $pd_R(P)\leq n$.
\end{enumerate}
\end{proposition}
\begin{proof} $1\Rightarrow 2.$ Clear since for every $n$-SG ring
$R$ we have $Ggldim(R)\leq n$ (by \cite[Proposition
2.2(1)]{MahdouTamekkante}).\\
$2\Rightarrow 3.$ Follows directly from \cite[Corollary
2.7]{Bennis
and Mahdou2}.\\
$3\Rightarrow 1.$ Follows from \cite[Proposition
2.10]{MahdouTamekkante}.
\end{proof}
Similarly for the $n$-wSG rings we have the following result:

\begin{proposition}\label{prop2} Let $R$ be a ring, $n$ a positive integer, and consider
the following assertions:
\begin{enumerate}
    \item $R$ is $n$-wSG ring.
    \item $wGgldim(R)\leq n$ and for every $R$-module $M$ there
    exists a short exact sequence $0\rightarrow M \rightarrow
    F\rightarrow M \rightarrow 0$ where $fd_R(F)<\infty$.
    \item $wGgldim(R)<\infty$ and for every $R$-module $M$ there
    exists a short exact sequence $0\rightarrow M \rightarrow
    F\rightarrow M \rightarrow 0$ where $fd_R(F)\leq n$.
    \item For every $R$-module $M$ there
    exists a short exact sequence $0\rightarrow M \rightarrow
    F\rightarrow M \rightarrow 0$ where $fd_R(F)\leq n$.
\end{enumerate}
Then, $1\Leftrightarrow 2\Leftrightarrow 3 \Rightarrow 4$ with
equivalence if $R$ is coherent.
\end{proposition}
To prove this Proposition we need the following Lemmas:

\begin{lemma}\label{lemma wG}
Let $R$ be a coherent ring. The following assertions are
equivalent:
\begin{enumerate}
    \item $wGgldim(R)\leq n$
    \item $fd_R(I)\leq n$ for every injective $R$-module $I$.
\end{enumerate}
\end{lemma}
\begin{proof}
Follows by combining the equivalence \cite[Theorem
7$(1\Leftrightarrow 2)$]{Ding2} and the equality \cite[Theorem
3.7(1=2)]{Ding}.
\end{proof}
\begin{lemma}\label{lemma flat}
For every $R$-module $M$ we have $Gfd_R(M)\leq fd_R(M)$ with
equality if $fd_R(M)<\infty$.
\end{lemma}
\begin{proof} In first note that $Gfd_R(M)\leq m$ implies that
$Tor_R^i(M,I)=0$ for every injective $R$-module $I$ and each
$i>m$. Indeed, consider an $m$-step flat resolution of $M$ as
follows:
$$0\rightarrow G\rightarrow F_m\rightarrow ....\rightarrow
F_1\rightarrow M\rightarrow 0$$ Clearly $G$ is Gorenstein flat.
Then, by \cite[Theorem 3.6]{Holm}, $Tor_R^i(G,I)=0$ for every
injective $R$-module $I$ and each $i>0$. Therefore,
$Tor_R^{i+m}(M,I)=0$, as
desired.\\
 The first inequality of this Lemma  holds from the fact that every flat module
is Gorenstein flat. Now, suppose that $fd_R(M)=n>0$ and suppose by
absurd that $Gfd_R(M)<n$. Thus, there exists an $R$-module $N$
such that $Tor^n_R(M,N)\neq 0$. Note  that $N$ can not be
injective (since $Gfd_R(M)<n$). Hence, pick a short exact sequence
$0\rightarrow N \rightarrow I \rightarrow I/N\rightarrow 0$ where
$I$ is injective. Then, $0\neq Tor_R^n(M,N)=Tor^{n+1}_R(M,I/N)$.
Absurd since $fd_R(M)=n$. This contradiction finish the proof.
\end{proof}

\begin{proof}[Proof of Proposition \ref{prop2}]
$1\Rightarrow 2.$ Follows from the fact that every strongly
$n$-Gorenstein flat module has a Gorenstein flat dimension $\leq
n$ (by \cite[Proposition 3.2(1)]{MahdouTamekkante}).\\
$2\Rightarrow 3.$ Let $F$ be an $R$-module such that
$fd_R(M)<\infty$. We have $Gf_R(F)\leq n$ since $wGgldim(R)\leq
n$. Then, by Lemma \ref{lemma flat}, $fd_R(F)\leq n$. Hence, the
following implication is immediate.\\
$3\Rightarrow 1$ Let $M$ be an arbitrary $R$-module. For such
module there is an exact sequence $(\star)\quad 0\rightarrow
M\rightarrow F \rightarrow M \rightarrow 0$ where $fd_R(F)\leq n$.
On the other hand, $m:=Gfd_R(M)<\infty$ since $wGgldim(R)<\infty$.
Thus, from the note in the proof of Lemma \ref{lemma flat}, for
every injective $R$-module $I$ we have  $Tor^i_R(M,I)=0$ for all
$i>m$. So, from $(\star)$, we have
$Tor^{n+1}_R(M,I)=Tor^{n+2}(M,I)=...=Tor^{n+m+1}(M,I)=0$ since
$n+m+1>m$. Thus, $M$ is a strongly $n$-Gorenstein flat module, as
desired.\\
$3\Rightarrow 4.$ Obvious.\\
Now assume that $R$ is coherent.\\
$4\Rightarrow 1.$  Let $I$ be an injective $R$-module. By
hypothesis, there is an exact sequence $0\rightarrow I \rightarrow
F \rightarrow I\rightarrow 0$ where $fd_R(F)\leq n$. Clearly this
exact sequence splits. Thus, $I\oplus I\cong F$. Hence,
$fd_R(I)\leq n$. Consequently, from Lemma \ref{lemma wG} and since
$R$ is coherent, $wGgldim(R)\leq n$. Thus, for each $R$-module $M$
and every injective $R$-module $I$ we have $Tor_R^{n+1}(M,I)=0$
(since $Gfd_R(M)\leq n$). Thus, adding this fact to the hypothesis
condition, we conclude that every $R$-module is strongly
$n$-Gorenstein flat as desired.
\end{proof}
\begin{remark} Let $R$ be an $n$-SG ring (resp., $n$-wSG ring).
We have:
\begin{enumerate}
    \item From Propositions \ref{prop1} and
    \ref{prop2}, $Ggldim(R)\leq n$ (resp., $wGgldim(R)\leq
    n$).
    \item Using \cite[Corollary 1.2]{Bennis and Mahdou2},
    $gldim(R)\leq n$ (resp., $wdim(R)\leq n$) if, and only if, $wdim(R)<\infty$.
\end{enumerate}
\end{remark}

Recall that a ring $R$ is called perfect if, every flat $R$-module
is projective.

\begin{theorem}\label{thm1} Every $n$-SG ring is $n$-wSG with equivalence in the following two cases:
\begin{enumerate}
    \item  $R$ is Noetherian.
    \item $R$ is perfect with finite Gorenstein global dimension.
\end{enumerate}
\end{theorem}

To prove this theorem we need the following Lemma.
\begin{lemma}\label{wg<g}
For any arbitrary ring $R$ we have: $wGgldim(R)\leq Ggldim(R)$
with equality in the following tow cases:
\begin{enumerate}
    \item  $R$ is Noetherian.
    \item $R$ is perfect with finite Gorenstein global dimension.
\end{enumerate}
\end{lemma}
\begin{proof}
To prove the desired inequality, the standard argument shows that
it suffices to prove that every Gorenstein projective $R$-module
is Gorenstein flat provided $Ggldim(R)<\infty$. So, let
$\mathbf{P}$ be a complete projective resolution. We have to prove
that $I\otimes_R \mathbf{P}$ is exact for every injective
$R$-module $I$. Since $Ggldim(R)<\infty$ and by \cite[Corollary
2.7]{Bennis and Mahdou2} we have $fd_R(I)<\infty$. Let
$I^{\ast}:=Hom_{\mathbb{Z}}(I,\mathbb{Q}/\mathbb{Z})$ be the
character of $I$. By \cite[Theorem 3.52]{Rotman},
$id_R(I^{\ast})=fd_R(I)<\infty$. Again, by \cite[Corollary
2.7]{Bennis and Mahdou2}, $pd_R(I^{\ast})$ is finite.
Consequently, by \cite[Proposition 2.3]{Holm},
$Hom_R(\mathbf{P},I^{\ast})$ is exact. By adjointness,
$Hom_{\mathbb{Z}}(I\otimes_R\mathbf{P},\mathbb{Q}/\mathbb{Z})=
Hom_R(\mathbf{P},I^{\ast})$. Then, $I\otimes_R\mathbf{P}$ is
exact, as desired.\\
If $R$ is Noetherian, the converse inequality follows from the
equivalence  \cite[Theorem 12.3.1$(3\Leftrightarrow 4)$]{Enocks and janda}.\\
Now suppose that $R$ is perfect with $n:=Ggldim(R)<\infty$. We
prove that every Gorenstein flat $R$-module is Gorenstein
projective. Let $M$ be an arbitrary Gorenstein flat module. By
definition we can pick an $n$-step right flat resolution as
follows:
$$0\rightarrow M \rightarrow F_1\rightarrow
F_1\rightarrow...\rightarrow F_n\rightarrow G$$ where all $F_i$
are flat and so projective since $R$ is perfect. But $Gpd_R(G)\leq
n$. Thus, using the equivalence \cite[Theorem
2.20($i\Leftrightarrow iv$)]{Holm}, we conclude that $M$ is
Gorenstein projective as desired. Consequently, $wGgldim(R)\leq
Ggldim(R)$ and this finish the proof.
\end{proof}
\begin{proof}[Proof of Theorem \ref{thm1}]
Let $R$ be an $n$-SG ring. Clearly $Ggldim(R)\leq n$. Then, by
Lemma \ref{wg<g}, $wGgldim(R)\leq n$. Now, let $M$ be an arbitrary
$R$-module.  By hypothesis, $M$ is strongly $n$-Gorenstein
projective. Then, there is a short exact sequence $0\rightarrow M
\rightarrow P \rightarrow M \rightarrow 0$ where $pd_R(M)\leq n$.
So, by Proposition \ref{prop2}, $R$ is an $n-wSG$ ring, as
desired.\\
Now, let $R$ be an $n$-wSG ring. Then, $wGgldim(R)\leq n$. So, if
$R$ is Noetherian or perfect with finite Gorenstein global
dimension then $Ggldim(R)\leq n$ (by Lemma \ref{wg<g}). Now, let
$M$ be an arbitrary $R$-module. By hypothesis, there is an exact
sequence $\rightarrow M\rightarrow F \rightarrow M \rightarrow 0$
where $fd_R(F)\leq n$. Using \cite[Corollary 2.7]{Bennis and
Mahdou2}, we have $pd_R(F)\leq n$. So, by Proposition \ref{prop1},
$R$ is an $n$-SG ring, as desired.
\end{proof}

\begin{theorem}\label{product}
Let $\displaystyle\{R_i\}_{i=1}^m$ be a family of rings and set
$R:=\displaystyle\prod_{i=1}^m R_i$. Then,
 $R$ is an $n$-SG ring if, and only if, $R_i$ is an $n$-SG ring
    for each $i\in I$.\\
    Moreover, if  $R_i$ is Coherent for each $i$, then $R$ is an $n$-wSG ring if, and only if, $R_i$ is an $n$-wSG ring
    for each $i\in I$.
\end{theorem}
\begin{proof}
By induction on $m$ it suffices to prove the assertion for $m=2$.
First suppose that $R_1\times R_2$ is $n$-SG ring. We claim that
$R_1$ is an $n$-SG ring. Let $M$ be an arbitrary $R_1$ module.
$M\times 0$ can be viewed as an $R_1\times R_2$-module. For such
module and since $R_1\times R_2$ is an $n$-SG ring, there is an
exact sequence $0\rightarrow M\times 0\rightarrow P \rightarrow
M\times 0\rightarrow 0$ where $pd_{R_1\times R_2}(P)\leq n$. Thus,
since $R_1$ is a projective $R_1\times R_2$ module, by applying
$-\otimes_{R_1\times R_2}R_1$ to the sequence above, we find the
exact sequence of $R$-modules: $0\rightarrow M\times
0\otimes_{R_1\times R_2}R_1\rightarrow P\otimes_{R_1\times
R_2}R_1\rightarrow M\times 0\otimes_{R_1\times R_2}R_1\rightarrow
0$. Clearly $pd_{R_1}(P\otimes_{R_1\times R_2}R_1)\leq
pd_{R_1\times R_2}(P)\leq n$. Moreover, we have the isomorphism of
$R$-modules: $M\times 0\otimes_{R_1\times R_2}R_1\cong M\times
0\otimes_{R_1\times R_2}(R_1\times R_2)/(0\times R_2)\cong M$.
Thus, we obtain an exact sequence of $R$-module with the form:
$0\rightarrow M\rightarrow P\otimes_{R_1\times R_2}R_1\rightarrow
M\rightarrow 0$. On the other hand, by \cite[Theorem 3.1]{Bennis
and Mahdou3}, we have $Ggldim(R_1)\leq Ggldim(R_1\times R_2)\leq
n$. Thus, using
Proposition \ref{prop1}, $R_1$ is an $n$-SG ring, as desired.\\
By the same argument, $R_2$ is also an $n$-SG ring.\\
Now, suppose that $R_1$ and $R_2$ are $n$-SG rings and we claim
that $R_1\times R_2$ is an $n$-SG ring. Let $M$ be an $R_1\times
R_2$-module. We have $$M\cong M\otimes_{R_1\times R_2}(R_1\times
R_2)\cong M\otimes_{R_1\times R_2}((R_1\times 0)\oplus (R_2\times
0))\cong M_1\times M_2$$ where $M_i=M\otimes_{R_1\times R_2}R_i$
for $i=1,2$. For each $i=1,2$, there is an exact sequence
$0\rightarrow M_i\rightarrow P_i\rightarrow M_i\rightarrow 0$
where $pd_{R_i}(P_i)\leq n$ since $R_i$ is an $n$-SG ring. Thus,
we have the exact sequence of $R_1\times R_2$-modules:
$0\rightarrow M_1\times M_2\rightarrow P_1\times P_2\rightarrow
M_1\times M_2\rightarrow 0$. On the other hand, $pd_{R_1\times
R_2}(P_1\times P_2)=sup\{pd_{R_i}(P_i)\}_{1,2}\leq n$ (by
\cite[Lemma 2.5(2)]{Mahdou costa}). Moreover, from \cite[Theorem
3.1]{Bennis and Mahdou3}, $Ggldim(R_1\times
R_2)=sup\{Ggldim(R_i)\}_{1,2}\leq n$. Thus, from Proposition
\ref{prop1}, $R_1\times R_2$ is an $n$-SG ring, as desired.\\
If $R_i$ is coherent for each $i=1,2$, by \cite[Thoerem
2.4.3]{Glaz}, $R_1\times R_2$ is coherent. So, using \cite[Theorem
3.5 and Lemma 3.7]{Bennis and Mahdou3}, by the same reasoning that
above, we prove the result for the $n$-wSG rings.
\end{proof}
Let $T:=R[X_1,X_2,...,X_n]$ the polynomial ring in $n$
indeterminates over $R$. If we suppose that $T$ is an $m$-SG ring,
it is easy to see by \cite[Theorem 2.1]{Bennis and Mahdou3}, that
$n\leq m$.

\begin{theorem}\label{R[X]1}
If $R[X_1,X_2,...,X_n]$ is an $m$-SG ring, then $R$ is an
$(m-n)$-SG ring.
\end{theorem}
\begin{proof}
By induction on $n$ is suffices to prove the result for $n=1$. So,
suppose that $R[X]$ is an $m$-SG ring. Let $M$ be an arbitrary
$R$-module. For the $R[X]$-module $M[X]:=M\otimes_RR[X]$ there is
an exact sequence of $R[X]$-modules $0\rightarrow M[X]\rightarrow
P \rightarrow M[X]\rightarrow 0$ where $pd_{R[X]}(P)\leq m$.
 Applying $-\otimes_{R[X]}R$ to the short exact sequence above and
 seeing that $M\cong_RM[X]\otimes_{R[X]}R$, we obtain a short exact
 sequence of $R$-modules with the form $0\rightarrow M \rightarrow P\otimes_{R[X]}R\rightarrow M \rightarrow 0$ (see that $R$ is a projective
 $R[X]$-module). Moreover, $pd_R(P\otimes_{R[X]}R)\leq
 pd_{R[X]}(P)<\infty$. On the other hand,
 $Ggldim(R)=Ggldim(R[X])-1\leq m-1$ (by \cite[Theorem 2.1]{Bennis and Mahdou3}). Hence, by Proposition
 \ref{prop1}, $R$ is an $(m-1)$-SG ring, as desired.
\end{proof}
\begin{theorem}\label{R[X]2}
If $R[X_1,X_2,...,X_n]$ is a coherent $m$-wSG ring, then $R$ is an
$(m-n)$-wSG-ring.
\end{theorem}
\begin{proof} Note in first that for every $i\leq n$, the
polynomial ring $R[X_1,...,X_i]$ is coherent (by using
\cite[Theorem 4.1.1(1)]{Glaz}). By the induction on $n$, it
suffices to prove the result for $n=1$. Using \cite[Theorem
2.11]{Bennis and Mahdou3} and Proposition \ref{prop2}, the proof
is similar to the proof of Theorem \ref{R[X]1}.
\end{proof}

Trivial examples of the $n$-SG-ring (resp., $n$-wSG ring) are the
rings with global dimension (resp., weak global dimension) $\leq
n$. The following example give a new family of commutative $n$-SG
rings (resp., $n$-wSG rings) with infinite weak global dimension.
\begin{example}\label{1}
Consider the non semi-simple quasi-Frobenius rings
$R_1:=K[X]/(X^2)$ and $R_2:=K[X]/(X^3)$ where $K$ is a field and
let $S$ be a non Noetherian  ring such that $gldim(S)= n$. Then,
\begin{enumerate}
\item $Ggldim(R_1)=Ggldim(R_2)=0$ and $R_1$ is $0$-SG ring but
$R_2 $ is not.
    \item $R_1\times S$ is a non Noetherian $n$-SG (and so $n$-wSG) ring with infinite weak
global dimension.
 \item $Ggldim(R_2\times S)=n$ but $R_2\times S$ is not an $n$-SG ring (a non Noetherian
 ring)with infinite weak
global dimension.
    \item $wGgldim(R_2[X_1,...,X_n])=Ggldim(R_2[X_1,...,X_n])=n$  but $R_2[X_1,...,X_n]$  is
    neither  $n$-SG ring nor $n$-wSG ring with infinite weak
global dimension.
\end{enumerate}

\end{example}

\begin{proof}
From \cite[Corollary 3.9]{ouarghi} and \cite[Proposition
2.6]{Bennis and Mahdou2}, $Ggldim(R_1)=Ggldim(R_2)=0$ and $R_1$ is
$0$-SG ring but $R_2 $ is not. So, $(1)$ is clear.  Moreover $R_1$
and $R_2$ have infinite weak global dimensions. By \cite[Theorems
2.1 and 3.1]{Bennis and Mahdou3} and Lemma \ref{wg<g} and the fact
that $R_2$ is Noetherian, it is easy to see that,
\begin{itemize}
    \item $Ggldim(R_2\times S)=n$, and
    \item $wGgldim(R_2[X_1,...,X_n])=Ggldim(R_2[X_1,...,X_n])=n$.
\end{itemize}
And using Theorems \ref{product} and \ref{R[X]1}, we conclude that
$R_1\times S$ is an  $n$-SG (and so $n$-wSG by Theorem \ref{thm1})
and that $R_2[X_1,...,X_n]$  is neither  $n$-SG ring nor $n$-wSG
ring (by Theorem \ref{thm1} since $R_2$ is Noetherian). Hence,
$(2)$, $(3)$ and $(4)$ hold.
\end{proof}

With rings with finite weak global dimension, it is clear that
left implication of Theorem \ref{thm1} is not true. The following
Example shows the same thing  with rings with infinite weak global
dimensions.

\begin{example}\label{3}
Consider the non semi-simple quasi-Frobenius rings
$R_1:=K[X]/(X^2)$ and $R_2:=K[X]/(X^3)$ where $K$ is a field and a
family of coherent rings $\{S_n\}_{n\in \mathbb{N}}$ such that
$n=wdim(S_n)<gldim(S_n)$ (for example $S_n:=S_0[X_1,X_2,...,X_n]$
where $S_0$ is a non-Noetherian Von Neumann regular ring). For
every positive integer n, set $R_1^n:=R_0\times S_n$ and
$R_2^n:=R_1\times S_n$. Then,
\begin{enumerate}
    \item $R_0^n$ is an $n$-wSG ring which is not $n-SG$ ring.
    \item $wGgldim(R_1^n)=n$ but $R_1^n$ is not an $n$-wSG ring.
\end{enumerate}

\end{example}

\begin{proof}
$(1).$ Since $wdim(S_n)=n$, the ring $S_n$ is $n-wSG$. On the
other hand, from \cite[Corollary 3.9]{ouarghi}, $R_1$ is a $0-SG$
ring. Then, it is also an $n-SG$ ring. Hence, from Theorem
\ref{thm1}, $R_1$ is an $n-wSG$ ring. Thus, by Theorem
\ref{product}, $R_1\times S_n$ is an $n-wSG$ ring. But
$gldim(S_n)>n$ implies that $Ggldim(R_1^n)>n$ (by \cite[Theorem
3.1]{Bennis and Mahdou3}). So, from Proposition \ref{prop1},
$R_1^n$ is not an $n$-SG ring, as
desired.\\
$(2).$ From \cite[Theorem 3.5]{Bennis and Mahdou3}, it is clear
that $wGgldim(R_2^n)=n$. And using Theorem \ref{product}, if
$R_2^n$ is an $n$-wSG ring, we conclude that $R_2$ is an $n$-wSG
ring and so $n$-SG ring since $R$ is Noetherian. But
$Ggldim(R_2)=0$. So, from Proposition \ref{prop1}, $R_2$ is a
$0$-SG ring. Contradiction with Example \ref{1}.
\end{proof}



\bigskip\bigskip

\end{document}